\documentclass[10pt]{amsart}
\usepackage{latexsym,enumerate}
\usepackage{xypic}
\usepackage{amsmath,amsthm,amsopn,amstext,amscd,amsfonts,amssymb}
\usepackage[dvips]{graphicx}
\usepackage[latin1]{inputenc}
\usepackage{epic}
\usepackage{curves}
\usepackage[dvips]{graphicx}
\usepackage[active]{srcltx}
\newcommand{\dyle}{\displaystyle}

\newcommand{\dint}{\dyle\int}

\newcommand{\iy}{\infty}
\newcommand{\p}{\partial}
\newcommand{\re}{{I\!\!R}}
\newcommand{\ren}{\re^N}

\newcommand{\irn}{\int\limits_{\re^N}}

\newcommand{\inn}{\mbox{ in }}

\renewcommand{\a }{\alpha }

\newcommand{\D }{\Delta }

\newcommand{\g }{\gamma}

\renewcommand{\l }{\lambda }

\newcommand{\s }{\sigma }

\renewcommand{\t }{\tau }

\renewcommand{\O }{\Omega }

\newcommand{\cqd}{{\unskip\nobreak\hfil\penalty50
        \hskip2em\hbox{}\nobreak\hfil\mbox{\rule{1ex}{1ex} \qquad}
        \parfillskip=0pt \finalhyphendemerits=0\par\medskip}}
        \newtheorem{Theorem}{Theorem}[section]

\newtheorem{Definition}[Theorem]{Definition}
\newtheorem{Lemma}[Theorem]{Lemma}
\newtheorem{Proposition}[Theorem]{Proposition}
\newtheorem{remarks}[Theorem]{Remarks}
\newtheorem{remark}[Theorem]{Remark}
\date{}
\catcode`@=11 \@addtoreset{equation}{section}
\renewcommand\theequation{\thesection.\@arabic\c@equation}
\catcode`@=12
\begin{document}

\title[ Fractional heat equation involving Hardy Potential]
{A note on the Fujita exponent in Fractional heat equation involving the Hardy potential}
\thanks{Work partially supported by Project MTM2016-80474-P, MINECO, Spain. The first author is also partially supported by an Erasmus grant from Autonoma University of Madrid.}
\author[B. Abdellaoui, I. Peral, A. Primo ]{Boumediene abdellaoui, Ireneo Peral, Ana Primo }

\address{\hbox{\parbox{5.7in}{\medskip\noindent {B. Abdellaoui, Laboratoire d'Analyse Nonlin\'eaire et Math\'ematiques
Appliqu\'ees. \hfill \break\indent D\'epartement de Math\'ematiques, Universit\'e Abou Bakr Belka\"{\i}d, Tlemcen, \hfill\break\indent Tlemcen 13000,
Algeria.}}}}
\address{\hbox{\parbox{5.7in}{\medskip\noindent{I. Peral and A. Primo, Departamento de Matem\'aticas,\\ Universidad Aut\'onoma de Madrid,\\
        28049, Madrid, Spain. \\[3pt]
        \em{E-mail addresses: }\\{\tt boumediene.abdellaoui@uam.es, \tt ireneo.peral@uam.es, ana.primo@uam.es
         }.}}}}

\date{\today}


 \begin{abstract}

 In this work, we are interested on the study of the Fujita exponent and the meaning of the blow-up for
 the Fractional Cauchy problem with the Hardy potential, namely,
\begin{equation*}
u_t+(-\Delta)^s u=\lambda\dfrac{u}{|x|^{2s}}+u^{p}\inn\ren,\\ u(x,0)=u_{0}(x)\inn\ren,
\end{equation*}
where $N> 2s$, $0<s<1$,  $(-\Delta)^s$ is the fractional laplacian of order $2s$, $\l>0$,  $u_0\ge 0$, and  $1<p<p_{+}(s,\lambda)$, where $p_{+}(\lambda, s)$ is
the critical existence power found in \cite{BMP} and \cite{AMPP}.
\end{abstract}

\maketitle

\rightline{\it To Sandro in his 70th birthday with  our friendship.}

\section{Introduction}\label{sec:cauchy}
In the pioneering work \cite{F}, Fujita found a critical exponent for the heat equation with  a semilinear term of power type.  More precisely, for the
problem,
\begin{equation}\label{Fujita}\left\{
\begin{array}{l}
u_t=\Delta u+u^p,\, x\in\ren,\, t>0,\\ u(x,0)=u_0(x)\geq 0,\,x\in\ren,
\end{array}\right.
\end{equation}
where $1<p<\iy$, Fujita proved that if $1<p< 1+\dfrac{2}{N}$, then there exists $T>0$ such that the solution to problem \eqref{Fujita} satisfies
$||u(\cdot,t_n)||_{\infty}\to \iy$ as $t_n\to T$. However, if $p>1+\dfrac{2}{N}$, then there are both global solutions for small data as well as non-global
solutions for large data. The critical value $F(0)=1+\dfrac{2}{N}$ is often called the critical Fujita blow-up exponent for the heat equation. Moreover it is
proved  that for $p=1+\dfrac{2}{N}$, a suitable norm of the solution goes to infinity in a finite time. We refer to \cite{W} for a simple proof of this last fact
(see also \cite{KST}).

Sugitani in \cite{SS} studies the same kind of question for \textit{the fractional heat equation}, that is, the problem,
\begin{equation}\label{nolocal}
\left\{
\begin{array}{rcll}
 u_t+(-\Delta)^{s} u&=& u^p &\mbox{ in } \Omega\times (0,T),\\
u(x,t)&>& 0 & \inn \Omega\times (0,T),\\ u(x,t)&=& 0 &\inn (\ren\setminus\Omega)\times[ 0,T),\\ u(x,0)&=& u_0(x) &\mbox{ if }x\in\O,
\end{array}
\right.
\end{equation}
where $N> 2s$, $0<s<1$,  $p>1$ and  $u_0\ge 0$ is  in a suitable class of functions.

By $(-\Delta)^s$ we denote the  fractional Laplacian of order $2s$ introduced by M. Riesz in \cite{MRiesz}, that is,
\begin{equation}\label{fraccionario}
(-\Delta)^{s}u(x):=a_{N,s}\mbox{ P.V. }\int_{\mathbb{R}^{N}}{\frac{u(x)-u(y)}{|x-y|^{N+2s}}\, dy},\, s\in(0,1),
\end{equation}
where
$$a_{N,s}=2^{2s-1}\pi^{-\frac N2}\frac{\Gamma(\frac{N+2s}{2})}{|\Gamma(-s)|}$$
is the normalization constant to have the identity
$$(-\Delta)^{s}u=\mathcal{F}^{-1}(|\xi|^{2s}\mathcal{F}u),\, \xi\in\mathbb{R}^{N}, s\in(0,1),$$
 for every $u\in \mathcal{S}(\mathbb{R}^N)$, the Schwartz class.
See \cite{Landkof}, \cite{dine}, \cite{FLS} and Chapter 8 of \cite{Peral-Soria}, for technical details and properties of the fractional Laplacian.

In \cite{APP},  the authors deal with the following problem,
\begin{equation}\label{local-Hardy}
\left\{
\begin{array}{rcll}
 u_t-\Delta u&=&\l\dfrac{\,u}{|x|^{2}}+ u^p+ c f &\mbox{ in } \Omega\times (0,T),\\
u(x,t)&>&0 & \inn \Omega\times (0,T),\\ u(x,t)&=&0 & \inn \p\Omega\times[ 0,T),\\ u(x,0)&=&u_0(x) &\mbox{ if }x\in\O,
\end{array}
\right.
\end{equation}
where $N>2$ and $0\in \Omega$.

This problem is related to the classical Hardy inequality:

\textsc{Hardy Inequality.} \textit{Assume $N\ge 3$. For all $\phi\in\mathcal{C}^\infty_0(\mathbb{R}^N)$ the following inequality holds,
\begin{equation}\label{HLD}
\Big(\frac{N-2}{2}\Big)^2\int_{\mathbb{R}^N}\dfrac{\phi^2(x)}{|x|^2} dx\le \int_{\mathbb{R}^N}|\nabla \phi(x)|^2dx.
\end{equation}
Moreover $\Lambda_N:=\Big(\dfrac{N-2}{2}\Big)^2$ is optimal and is not achieved. }

 The blow-up in the $L^\infty$ norm for the solution of problem \eqref{local-Hardy} is produced in any time $t>0$, for any nonnegative data and for all $p>1$, according with the results by Baras-Goldstein in \cite{BaGo}. Therefore, the Fujita behavior in the presence of the Hardy potential must be understood in a different way.

 For $\lambda>0$, setting $\mu_{1}(\l)=\frac{N-2}{2}-\sqrt{\Big(\frac{N-2}{2}\Big)^2-\l}$, then it was proved that if $1<p<1+\frac{2}{N-\mu_1(\l)}$, there exists $T^*>0$ that is independents of the nonnegative initial datum, such that the solution $u$ to problem \eqref{local-Hardy} satisfies
\begin{equation}\label{blow-up-local}
\lim\limits_{t\to T^{*}}\dint_{B_{r}(0)}|x|^{-\mu_1(\l)}u(x,t)\,dx=\iy,
\end{equation}
for any ball $B_{r}(0)$.  Moreover for $p>1+\dfrac{2}{N-\mu_1(\l)}$, if the initial datum is small enough, there exists a global solution to \eqref{local-Hardy}.  According to this behavior the corresponding {\it Fujita type exponent} for problem \eqref{local-Hardy} is defined by $F(\lambda)= 1+\dfrac{2}{N-\mu_1(\l)}$ and
the blow-up is understood in the sense of local weighted $L^1$ associated to \eqref{blow-up-local}. The Hardy inequality is an expression of the
\textit{uncertainty Heisenberg principle}, hence we can say that the result in  \cite{APP}, explains the influence of the uncertainty principle on the  diffusion
problem  \eqref{local-Hardy}.

The following fractional Hardy inequality appears  in \cite{FLS} in order to study the relativistic stability of the matter.
\begin{Theorem}\label{DH}{\it (Fractional Hardy inequality).}
For all $u\in \mathcal{C}^{\infty}_{0}(\ren)$ the following inequality holds,
\begin{equation}\label{Hardy}
\dint_{\ren} \,|\xi|^{2s} |\hat{u}|^2\,d\xi\geq \Lambda_{N,s}\,\dint_{\ren} |x|^{-2s} u^2\,dx,
\end{equation}
where
$$
\Lambda_{N,s}= 2^{2s}\dfrac{\Gamma^2(\frac{N+2s}{4})}{\Gamma^2(\frac{N-2s}{4})}.
$$
The constant $\Lambda_{N,s}$ is optimal and is not attained. Moreover, $\Lambda_{N,s}\to \Lambda_{N,1}:=\left(\dfrac{N-2}{2}\right)^2$, the classical Hardy
constant, when $s$ tends to $1$.
\end{Theorem}
This inequality was proved in  \cite{He}. See also \cite{B}, \cite{FLS}, \cite{SW} and \cite{Y}. The reader can find all the details of  a direct  proof in
Section 9.2 of \cite{Peral-Soria}.

Recently, in \cite{AMPP} and  related to the Hardy inequality stated in \eqref{Hardy}, the authors study the fractional parabolic semilinear problem,
\begin{equation}\label{local}
\left\{
\begin{array}{rcll}
 u_t+(-\Delta)^{s} u &=&\l\dfrac{\,u}{|x|^{2s}}+ u^p+ c f &\mbox{ in } \Omega\times (0,T),\\
u(x,t)&>&0 & \inn \Omega\times (0,T),\\ u(x,t)&=& 0 & \inn (\ren\setminus\Omega)\times[ 0,T),\\ u(x,0)&=&u_0(x) &\mbox{ if } x\in\O,
\end{array}
\right.
\end{equation}
where $N> 2s$, $0<s<1$,  $p>1$, $c,\l>0$, and  $u_0\ge 0$, $f\ge 0$ are in a suitable class of functions. By $(-\Delta)^s$ we denote the  fractional Laplacian of
order $2s$, defined in \eqref{fraccionario}. In \cite{AMPP} and \cite{BMP}, the authors prove the existence of a critical power $p_{+}(s,\lambda)$ such that if $p>
p_+(s,\lambda)$, the problem \eqref{local} has no weak positive supersolutions and a phenomenon of \emph{complete and instantaneous blow up} happens. If $p<
p_+(s,\lambda)$, there exists a positive solution  for a suitable class of nonnegative data.

In this note, we deal with the corresponding Fractional Cauchy problem,
\begin{equation}\label{Cauchy}
u_t+(-\Delta)^s u=\lambda\dfrac{u}{|x|^{2s}}+u^{p}\inn\ren\times(0,\infty),\\ u(x,0)=u_{0}(x)\inn\ren,
\end{equation}
with $1<p<p_{+}(s,\lambda)$ in order to find the value of the corresponding Fujita exponent.

The case $s\in (0,1)$ and $\l=0$ was considered in \cite{SS}. The author was able to show that $F(s):=1+\frac{2s}{N}$ is the associated Fujita exponent. See also
\cite{GK} for some extensions.

For $\lambda>0$, any solution to problem \eqref{local} is unbounded close to the origin, even for nice data (see \cite{AMPP}). This is the corresponding nonlocal
version of the Baras-Goldstein results for the heat equation developed in \cite{BaGo}.  Therefore, $L^\infty$-blow-up is instantaneous and free in problem
\eqref{Cauchy} as in the local case and  the blow-up will be also obtained in a suitable Lebesgue space with a weight.

In this work we will treat the case $s\in (0,1)$ and $\l>0$ that is more involved than the local problem for several reasons, one of them that the kernel of the
fractional heat equation has not a closed form with the exception of $s=\frac12$ and $s=1$.

The paper is organized as follows. In Section \ref{sec2} we introduce some tools about the fractional equation. The Fujita exponent  $F(\lambda, s)$ for problem \eqref{Cauchy} in obtained in Section \ref{sub}. Notice that by the Fujita exponent, we understand that, independently of the initial datum, for $1<p<F(\lambda, s)$, any solution to \eqref{Cauchy} blows-up in a certain weighted norm in a finite time. The critical case $p=F(\l,s)$ is analyzed in Subsection \ref{sub1}. In this case we are able to show a blow-up of a precise norm of $u$ that reflects the critical exponent $F(\l,s)$. In Section \ref{global}, for $F(\lambda, s)<p<p_{+}(\lambda, s)$, we prove the existence
of global solutions for suitable data. This shows in some sense the optimality of our blow up results.

\section{Preliminaries tools}\label{sec2}

First, we enunciate some Lemmas and notations that we will use along the paper (see \cite{AMPP} for a proof).
\begin{Lemma} \label{singularity} Let $0<\lambda\leq \Lambda_{N,s}$. Then $v_{\pm\alpha}=|x|^{-\frac{N-2s}{2}\pm\alpha_{\lambda}}$ are
solutions to
\begin{equation}\label{eq:homogenous}
(-\Delta)^s u= \lambda\frac{u}{|x|^{2s}}\inn (\ren\setminus{\{0\}}),
\end{equation}
where $\alpha_{\lambda}$ is obtained by the identity
\begin{equation}\label{lambda}
\lambda=\lambda(\alpha_{\lambda})=\lambda(-\alpha_{\lambda})=\dfrac{2^{2s}\,\Gamma(\frac{N+2s+2\alpha_{\lambda}}{4})\Gamma(\frac{N+2s-2\alpha_{\lambda}}{4})}{\Gamma(\frac{N-2s+2\alpha_{\lambda}}{4})\Gamma(\frac{N-2s-2\alpha_{\lambda}}{4})}.
\end{equation}
\end{Lemma}
\begin{remark}
Notice that $\lambda(\alpha_{\lambda})= \lambda(-\alpha_{\lambda})=m_{\alpha_{\lambda}}m_{-\alpha_{\lambda}}$, with $m_{\alpha_{\lambda}}=
2^{s}\dfrac{\Gamma(\frac{N+2s+2\alpha_{\lambda}}{4})}{\Gamma(\frac{N-2s-2\alpha_{\lambda}}{4})}$.
\end{remark}

\begin{Lemma}\label{LambdaVsSing}
The following equivalence holds true:
$$
0<\lambda(\alpha_{\lambda})=\lambda(-\alpha_{\lambda})\leq \Lambda_{N,s}\mbox{ if and only if } 0\leq \alpha_{\lambda}<\dfrac{N-2s}{2}.
$$
\end{Lemma}

\begin{remark}\label{gamma1}
Notice that we can explicitly construct two positive solutions to the homogeneous problem \eqref{eq:homogenous}. Henceforth, we denote
\begin{equation}\label{g1}
\mu(\lambda)= \dfrac{N-2s}{2}-\alpha_{\lambda} \hbox{ and } \bar\mu(\lambda)= \dfrac{N-2s}{2}+\alpha_{\lambda},
\end{equation}
with $0<\mu(\lambda)\leq \dfrac{N-2s}{2}\leq\bar\mu(\lambda)<(N-2s)$. Since $N-2\mu(\lambda)-2s={2}\alpha_{\lambda}>0$ and
$N-2\bar\mu(\lambda)-2s=-{2}\alpha_{\lambda}<0$, then $|x|^{-\mu(\lambda)}$ is the unique nonnegative solution that is locally in the energy space.
\end{remark}

The critical existence power  $p_{+}(\lambda,s)$, found in \cite{BMP} and \cite{AMPP},  depends on $s$ and $\lambda$, and in particular satisfies:
$$
p_{+}(\lambda,s):= 1+\dfrac{2s}{\frac{N-2s}{2}-\alpha_{\lambda}}=1+\dfrac{2s}{\mu (\lambda)}.
$$
Note that if $\lambda=\Lambda_{N,s}$, namely, $\alpha_{\lambda}=0$, then $p_{+}(\lambda,s)= \dfrac{N+2s}{N-2s}=2^{*}_{s}-1$, and if $\lambda=0$, namely,
$\alpha_{\lambda}=\dfrac{N-2s}{2}$, then $p_{+}(\lambda,s)=\infty$. Noting
$$
p_{-}(\lambda,s)= 1+\dfrac{2s}{\frac{N-2s}{2}+\alpha_{\lambda}}=1+\dfrac{2s}{\bar{\mu} (\lambda)},
$$
it follows that for $\lambda=\Lambda_{N,s}$, namely, $\alpha_{\lambda}=0$, then $p_{-}(\lambda, s)= 2^{*}_{s}-1$ and if $\lambda=0$, namely,
$\alpha_{\lambda}=\dfrac{N-2s}{2}$, then $p_{-}(\lambda,s)=\dfrac{N}{N-2s}$. Hence,
$$
\dfrac{N}{N-2s}\leq p_{-}(\lambda, s)\leq 2^{*}_{s}-1\leq p_{+}(\lambda,s)\leq \infty.
$$

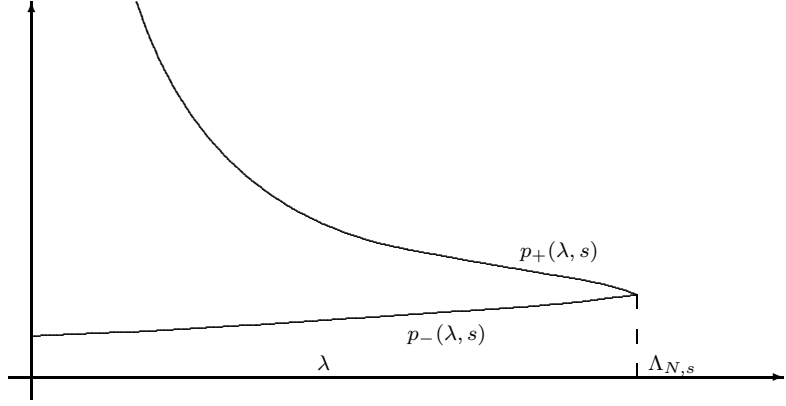
\begin{figure}
\begin{center}
\setlength{\unitlength}{1mm}
\begin{picture}(100,45)

\put(-3,0){\vector(1,0){103}} \put(0,-3){\vector(0,1){53}}

\put(65,16){\footnotesize$\dyle p_{+}(\lambda,s)$} \put(50,5){\footnotesize$\dyle p_{-}(\lambda,s)$} \put(38,1){\footnotesize$\dyle \lambda$}
\put(82,1){\footnotesize$\dyle \Lambda_{N,s}$}

\dashline[3]{2}(80.5,11)(80.5,0) \qbezier(0,5.5)(15,6)(30,7) \qbezier(30,7)(45,8)(60,9) \qbezier(60,9)(70,9.5)(80.5,11)

\qbezier(14,50)(23,22)(50,17) \qbezier(50,17)(60,15.14)(70,13.5) \qbezier(80.5,11)(77,12.5)(70,13.5)

\end{picture}
\end{center}
\caption{Fujita exponent for fractional Cauchy problem with Hardy potential.}
\end{figure}

\section{Blow up result for the Cauchy problem.}\label{sub}
It is clear that $L^\infty$-blow-up is instantaneous and free in problem \eqref{Cauchy} because the solutions are unbounded at the origin.

Before staring the main blow-up result we begin by precising the sense for which blow up is considered. As in the case $s=1$, $\l>0$, this phenomenon will be
analyzed in a suitable weighted Lebesgue space.
\begin{Definition}\label{def:blow-sense}
Consider $u(x, t)$ a positive solution to \eqref{Cauchy}, then we say that $u$  blows-up in a finite time if there exists $T^{*}<\iy$ such that
$$
\lim\limits_{t\to T^{*}}\dint_{\ren}|x|^{-\mu(\lambda)}u(x,t)\,dx=\iy,
$$
with $\mu(\lambda)= \dfrac{N-2s}{2}-\alpha_{\lambda}$.
\end{Definition}

The next proposition justifies in some sense the previous definition.
\begin{Proposition}
Let $\l\le \Lambda_{N,s}$ and consider $u$ to be a nonnegative solution to problem \eqref{Cauchy}, then
$$
\dint_{B_{r}(0)}|x|^{-\mu(\l)}u_{0}(x)\,dx<\iy, \hbox{  for some }r>0.
$$
In particular, for all $t\in (0,T)$, we have
$$
\dint_{B_{r}(0)}|x|^{-\mu(\l)}u(x,t)\,dx<\iy.
$$
\end{Proposition}
The proof follows combining the approximating arguments used in \cite{APP} and \cite{AMPP}.

The main blow up result of this section is the following.
\begin{Theorem}\label{blowup}
Suppose that $1<p<F(\l,s):=1+\dfrac{2s}{N-\mu(\l)}$ and let $u$ be a positive solution to problem \eqref{Cauchy}. Then there exists $T^*:=T^*(u_0)$ such that
$$
\lim\limits_{t\to T^{*}}\dint_{\ren}|x|^{-\mu(\l)}u(x,t)\,dx=\iy.
$$
\end{Theorem}
Before proving Theorem \ref{blowup}, we need some analysis related the fractional heat equation.

Let $h(x,t)$ be the \textit{fractional Heat Kernel}, namely,
$$
h_t+(-\Delta)^{s}h=0\inn\mathbb{R}^{N}\times (0,\infty),\, \,h(x,0)=\delta_{0}.
$$
There is no known closed form for $h(t,x)$ in real variables. However, in Fourier variables it is simply $\mathcal{F}({h})(t,\xi)=e^{-t|2\pi\xi|^{2s}}$. The
properties of the kernel $h$ were studied in \cite{Pol} for $N=1$ and in \cite{Bluget} for all dimensions. More precisely,

\textit{since $h(x, t)$ is defined by
\begin{equation}\label{kernel}
h(x, t)=\int_{\mathbb{R}^N} e^{2\pi  i\langle x,\xi\rangle}e^{-(2\pi|\xi|)^{2s}t}d\xi.
\end{equation}
where $0<s<1$ and  $N\ge 2s$, then there exists a constant  $C>1$ such that
\begin{equation}\label{P}
\frac{1}{C}\frac{1}{(1+|x|^{N+\alpha})}\leq h(x, 1)\leq \frac{C}{(1+|x|^{N+\alpha})}, \quad \hbox{ for all } x\in \mathbb{R}^N.
\end{equation}}

 There is   a direct  approach  inside of the Real Analysis field
 and even without using Bessel functions. Such a proof is based on a
 celebrated result  by S. N. Bernstein about the characterization of \textit{completely monotone functions}
 via Laplace transform. See Section 12.5 of \cite{Peral-Soria} for a detailed proof.

Notice that $h(x,t)=t^{-\frac{N}{2s}}H(\frac{|x|}{t^\frac{1}{2s}})$  and $H$ is a decreasing function that satisfies
$$
H(\s)\approx \frac{1}{(1+\s^2)^{\frac{N+2s}{2}}}, |H'(\s)|\le \frac{C}{(1+\s^2)^{\frac{N+2s+1}{2}}},
$$
with
$$
2s(-\D)^s H=NH+r H'.
$$
See for instance \cite{CF} and \cite{S}. We set
$$
\hat{h}(x,t)=\bigg(\frac{|x|}{t^\frac{1}{s}}\bigg)^{-\mu(\l)}h(x,t)\equiv
\bigg(\frac{|x|}{t^\frac{1}{s}}\bigg)^{-\mu(\l)}t^{-\frac{N}{2s}}H(\frac{|x|}{t^\frac{1}{2s}}).
$$
Notice that  we have the elementary formula,
$$
(-\D)^s (w\,v)=v(-\D)^sw +w(-\D)^sv-\irn\frac{(w(x)-w(y))(v(x)-v(y))}{|x-y|^{N+2s}}dy.
$$
Hence, for $t>0$, we have
\begin{eqnarray*}
(-\Delta)^{s}(\hat{h}(x,t)) &= & \bigg(\frac{|x|}{t^\frac{1}{s}}\bigg)^{-\mu(\l)}(-\Delta)^{s} h(x,t)+ h(x,t)(-\Delta)^{s}
\bigg(\frac{|x|}{t^\frac{1}{s}}\bigg)^{-\mu(\l)}\\ &-& \dyle \dint_{\mathbb{R}^{N}}
\dfrac{\Big(\bigg(\frac{|x|}{t^\frac{1}{s}}\bigg)^{-\mu(\l)}-\bigg(\frac{|y|}{t^\frac{1}{s}}\bigg)^{-\mu(\l)} \Big)\Big(h(x,t)-h(y,t)\Big)}{|x'-y|^{N+2s}} \,dy.
\end{eqnarray*}
Notice that
$$
\Big(\bigg(\frac{|x|}{t^\frac{1}{s}}\bigg)^{-\mu(\l)}-\bigg(\frac{|y|}{t^\frac{1}{s}}\bigg)^{-\mu(\l)}
\Big)\Big(h(x,t)-h(y,t)\Big)=t^{\frac{\mu(\l)}{s}-\frac{N}{2s}}\big(|x|^{-\mu(\l)}-|y|^{-\g}\big)\big(H(\frac{|x|}{t^\frac{1}{2s}})-H(\frac{|y|}{t^\frac{1}{2s}})\big)\ge
0.
$$
Thus
\begin{eqnarray*}
(-\Delta)^{s}(\hat{h}(x,t)) &\le & \bigg(\frac{|x|}{t^\frac{1}{s}}\bigg)^{-\mu(\l)}(-\Delta)^{s} h(x,t)+ h(x,t)(-\Delta)^{s}
\bigg(\frac{|x|}{t^\frac{1}{s}}\bigg)^{-\mu(\l)}\\ &= & \bigg(\frac{|x|}{t^\frac{1}{s}}\bigg)^{-\mu(\l)}(-h_t(x,t))+ \frac{\l
h(x,t)}{|x|^{2s}}\bigg(\frac{|x|}{t^\frac{1}{s}}\bigg)^{-\mu(\l)}\\ &= & \frac{N}{2s}
\frac{\hat{h}(x,t)}{t}+\frac{1}{2s}|x|t^{-\frac{N}{2s}-\frac{1}{2s}-1}H'\bigg(\frac{|x|}{t^\frac{1}{2s}}\bigg)\bigg(\frac{|x|}{t^\frac{1}{s}}\bigg)^{-\mu(\l)}+\frac{\l
\hat{h}(x,t)}{|x|^{2s}}\\ &\le & \frac{N}{2s} \frac{\hat{h}(x,t)}{t} +\frac{\l \hat{h}(x,t)}{|x|^{2s}},\\
\end{eqnarray*}
where we have used the fact that $H'\le 0$. Thus
 $$
-(-\Delta)^{s}(\hat{h}(x,t))+\frac{\l \hat{h}(x,t)}{|x|^{2s}}\ge -\frac{N}{2s} \frac{\hat{h}(x,t)}{t}.
 $$
We are now in position to prove Theorem \ref{blowup}.

{\bf Proof of Theorem \ref{blowup}.} We follow closely some arguments developed in \cite{APP}, see also \cite{QS}. Let $u$ be a positive solution to
\eqref{Cauchy}. Fix $\eta>0$ to be chosen later and define the function $\psi_\eta$
$$
\psi_\eta(x)=\hat{h}(x,\frac{1}{\eta})= \eta^{\frac{N}{2s}-\frac{\mu(\l)}{s}}|x|^{-\mu(\l)} H(\eta^\frac{|x|}{2s}),$$ then by the previous computation it holds
that
 $$
-(-\Delta)^{s} \psi_\eta(x)+\frac{\l \psi_\eta(x)}{|x|^{2s}}\ge -\frac{N}{2s}\eta \psi_\eta(x).
 $$
Notice that
$$
\irn \psi_\eta(x) dx=C\eta^{-\frac{\mu(\l)}{2s}}.
$$
Now, using $\psi_\eta$ as a test function in \eqref{Cauchy}, we get
$$
\dfrac{d}{dt}\irn u\psi_\eta dx=\irn u^p\psi_\eta dx +\irn \bigg(-(-\Delta)^{s} \psi_\eta(x)+\frac{\l \psi_\eta(x)}{|x|^{2s}}\bigg)u dx.
$$
Thus
$$
\dfrac{d}{dt}\irn u\psi_\eta dx\ge \irn u^p\psi_\eta dx -\frac{N}{2s} \eta \irn \psi_\eta(x)u dx.
$$
Using Jensen inequality, there results that
$$
\irn u^p\psi_\eta dx\ge C\eta^{(p-1)\frac{\mu(\l)}{2s}}\bigg(\irn u\psi_\eta dx\bigg)^p.
$$
Then
$$
\dfrac{d}{dt}\irn u\psi_\eta dx+\frac{N}{2s} \eta \irn \psi_\eta(x)u dx\ge C\eta^{(p-1)\frac{\mu(\l)}{2s}}\bigg(\irn u\psi_\eta dx\bigg)^p.
$$
Setting $$Y(t)=e^{\frac{N}{2s}\eta t}\irn u\psi_\eta dx,$$ it follows that
$$
Y'(t)\ge C \eta^{(p-1)\frac{\mu(\l)}{2s}}e^{-(p-1)\frac{N}{2s} \eta t}Y^p(t).
$$
Integrating the previous differential inequality, we arrive to
\begin{eqnarray*}
\frac{1}{p-1}\big(\frac{1}{Y^{p-1}(0)}-\frac{1}{Y^{p-1}(t)}\big) &\ge & C
\eta^{(p-1)\frac{\mu(\l)}{2s}}\frac{1}{(p-1)\frac{N}{2s}\eta}\big(1-e^{-(p-1)\frac{N}{2s} \eta t}\big)\\ &\ge &
\frac{C}{\frac{N}{2s}(p-1)}\eta^{(p-1)\frac{\mu(\l)}{2s}-1}\big(1-e^{-(p-1)\frac{N}{2s} \eta t}\big).
\end{eqnarray*}
Therefore,
$$
Y^{p-1}(t)\ge \dfrac{1}{\bigg(\frac{1}{Y^{p-1}(0)}-{C}{\frac{2s}{N}}\eta^{(p-1)\frac{\mu(\l)}{2s}-1}(1-e^{-(p-1)\frac{N}{2s} \eta t})\bigg)}.
$$
It is clear that, if for some $T<\infty$, we have
\begin{equation}\label{condim}
\frac{1}{Y^{p-1}(0)}\le {C}\frac{2s}{N}\eta^{(p-1)\frac{\mu(\l)}{2s}-1}(1-e^{-(p-1)\frac{N}{2s} \eta T}),
\end{equation}
then $Y(T)=\infty$.

Since $(1-e^{-(p-1)\frac{N}{2s}\eta T})\to 1$ as $T\to \infty$, then condition \eqref{condim} holds if
$$
Y^{p-1}(0)>\frac{1}{C}\frac{N}{2s}\eta^{1-(p-1)\frac{\mu(\l)}{2s}}.
$$
Hence
$$
\eta^{(p-1)(\frac{N}{2s}-\frac{\mu(\l)}{s})}\bigg(\irn u_0(x)|x|^{-\mu(\l)}H(\eta^\frac{1}{2s}
|x|)dx\bigg)^{p-1}>\frac{1}{C}\frac{N}{2s}\eta^{1-(p-1)\frac{\mu(\l)}{2s}},
$$
and then
\begin{equation}\label{main0}
\bigg(\irn u_0(x)|x|^{-\mu(\l)}H(\eta^\frac{1}{2s}|x|)dx\bigg)^{p-1}>\frac{1}{C}\frac{N}{2s} \eta^{-(p-1)(\frac{N}{2s}-\frac{\mu(\l)}{2s})+1}.
\end{equation}
It is clear that \eqref{main0} holds for $\eta$ small if and only if
$$
-(p-1)(\frac{N}{2s}-\frac{\mu(\l)}{2s})+1>0
$$
and then $p<F(\l,s)$.

Since
$$\irn u_0(x)|x|^{-\mu(\l)} dx>C_0,$$
using the fact that $H$ is bounded, there exists $\eta>0$ such that
\begin{equation}\label{main002}
\bigg(\irn u_0(x)|x|^{-\mu(\l)}H(\eta^\frac{1}{2s}|x|)dx\bigg)^{p-1}\ge 2s{N}{C}\eta^{-(p-1)(2s{N}{\beta}-\frac{\mu(\l)}{2s})+1}.
\end{equation}
Hence the result follows. \cqd

\subsection{The critical case}\label{sub1}

Notice that the above argument does not hold for the critical case $p=F(\l,s)$. Hence in this case we will use a different argument based on a suitable apriori
estimates as in \cite{MP} and \cite{GK}. More precisely we have

\begin{Theorem}\label{blowupp}
Assume that $p=F(\l,s):=1+\dfrac{2s}{N-\mu(\l)}$. If $u$ is a positive solution to problem \eqref{Cauchy}, then there exists $T^*:=T^*(u_0)$ such that
$$
\lim\limits_{t\to T^{*}}\dint_{\ren}|x|^{-\mu(\l)}u^p(x,t)\,dx=\iy.
$$
\end{Theorem}
\begin{proof} We will perform the \textit{ground state transform}, i.e.,
define $v(x,t):=|x|^{\mu(\l)} u(x,t)$, then
$$
(-\Delta)^{s} u-\l\dfrac{\,u}{|x|^{2s}}=|x|^{\mu(\l)} L v(x,t)
$$
where
$$
L (v(x,t)):= a_{N,s}\:\:p.v. \int_{\mathbb{R}^{N}} (v(x,t)-v(y,t))K(x,y)dy
$$
and
$$K(x,y)=\dfrac{1}{|x|^{\mu(\l)}}\dfrac{1}{|y|^{\mu(\l)}}\dfrac{1}{|x-y|^{N+2s}}.$$
See \cite{FLS} and \cite{AMPP}. Thus $v$ solves the parabolic equation
\begin{equation}\label{peso1}
\left\{
\begin{array}{rcll}
|x|^{-2\mu(\l)}v_t+L v &= & |x|^{-\mu(\l)}u^p=|x|^{-\mu(\l)(p+1)}v^p & \mbox{ in } \ren\times (0,T),\\ |x|^{-\mu(\l)}v(x,0)&=& u_0(x) & \mbox{ in } \ren.
\end{array}
\right.
\end{equation}
It is clear that
$$
\dint_{\ren}|x|^{-\mu(\l)}u^p(x,t)\,dx=\dint_{\ren}|x|^{-\mu(\l)(p+1)}v^p\,dx.
$$
Therefore, in order to show the  blow-up result we will prove that
$$
\lim\limits_{t\to T^{*}}\dint_{\ren}|x|^{-\mu(\l)(p+1)}v^p\,dx=\iy.
$$
We argue by contradiction. Assume that $\dint_{\ren}|x|^{-\mu(\l)(p+1)}v^p\,dx<\infty$ for all $t<\infty$. We claim that
\begin{equation}\label{apr}
\int_0^\infty\int_{\ren}|x|^{-\mu(\l)(p+1)}v^p dxdt\le C.
\end{equation}
Let $\varphi\in \mathcal{C}^\infty_0(\ren)$ be such that $0\le \varphi\le 1$, $\varphi=1$ in $B_1(0)$ and $\varphi=0$ in $\ren\backslash B_2(0)$. Define
$\psi(x,t)=\varphi(\frac{t^2+|x|^{4s}}{R^2})$ with $R>>1$. It is clear that if $t>R$, then $\psi(x,t)=0$. Fix $T>R$, then using $\psi^m$ as a test function in
\eqref{peso1}, with $1<m<p'$, setting $Q_T=\ren\times (0,T)$ and using Kato inequality, it holds that
\begin{equation}\label{inter0}
\begin{array}{lll}
&\dyle \iint_{Q_T} |x|^{-\mu(\l)(p+1)}v^p\psi^m dxdt +\int_{\ren}|x|^{-\mu(\l)}v(x,0)\psi^m(x,0)dx\\ &=\dyle  \iint_{Q_T} |x|^{-2\mu(\l)}v\bigg(-(\psi^m)_t dxdt +
L \psi^m\bigg)dxdt\\ &\le \dyle  m\iint_{Q_T} |x|^{-2\mu(\l)}v\psi^{m-1}(-\psi_t)dxdt+m\iint_{Q_T}v\psi^{m-1} L \psi dxdt=I+J.
\end{array}
\end{equation}
We begin by estimating $I$. Define
$$
Q_T^1=\bigg\{(x,t)\in Q_T\mbox{ such that } R^2<t^2+|x|^{4s}<2R^2\bigg\},
$$
$$
Q_T^2=\bigg\{(x,t)\in Q_T\mbox{ such that } t^2+|x|^{4s}<2R^2\bigg\},
$$
it is clear that $\text{supp}\psi_t\subset Q_T^1$ and $\text{supp}\psi\subset Q_T^1$. Then we have
\begin{eqnarray*}
I &\le & \dyle  m\iint_{Q_T} |x|^{-2\mu(\l)}v\psi^{m-1}|\psi_t| dxdt\le \dyle  m\iint_{Q_T^1} |x|^{-2\mu(\l)}v\psi^{m-1}|\psi_t| dxdt\\
 & \le & m\bigg(\iint_{Q_T^1} |x|^{-\mu(\l)(p+1)}v^p\psi^m dxdt\bigg)^{\frac{1}{p}} \bigg(\iint_{Q_T^1}
 |x|^{-\mu(\l)}\frac{|\psi_t|^{p'}}{\psi^{p'-m}}dxdt\bigg)^{\frac{1}{p'}}.
\end{eqnarray*}
In the same way we have
\begin{eqnarray*}
J &\le & \dyle  m\iint_{Q_T^2}v\psi^{m-1} |L \psi| dxdt\\ & \le & \bigg(\iint_{Q_T^2} |x|^{-\mu(\l)(p+1)}v^p\psi^m dxdt\bigg)^{\frac{1}{p}} \bigg(\iint_{Q_T^2}
|x|^{\frac{\mu(\l)(p+1)}{p-1}}\frac{|L \psi|^{p'}}{\psi^{p'-m}}dxdt\bigg)^{\frac{1}{p'}}.
\end{eqnarray*}
Now, since $p=F(\l,s)$ and setting $\t=\frac{t}{R}, y=\frac{x}{R^{\frac{1}{2s}}}$, we reach that
$$
\iint_{Q_T^1} |x|^{-\mu(\l)}\frac{|\psi_t|^{p'}}{\psi^{p'-m}}dxdt=2^{p'}\iint_{\{1<\t^2+|y|^{4s}<2\}}
|y|^{-\mu(\l)}\t^{p'}\frac{|\varphi'(\t^2+|y|^{4s})|^{p'}}{\varphi^{p'-m}(\tau^2+|y|^{4s})}dyd\t\equiv C_1,
$$
and
$$
\iint_{Q_T^2} |x|^{\frac{\mu(\l)(p+1)}{p-1}}\frac{|L \psi|^{p'}}{\psi^{p'-m}}dxdt=\iint_{\{\t^2+|y|^{4s}<2\}}|y|^{\frac{\mu(\l)(p+1)}{p-1}}\frac{|L
\theta(y,\t)|^{p'}}{\theta^{p'-m}}dyd\t=C_2,
$$
where $\theta(y,\t)=\varphi(\t^2+|y|^{4s})$. Thus
\begin{equation}\label{inter1}
\begin{array}{lll}
&\dyle \iint_{Q_T} |x|^{-\mu(\l)(p+1)}v^p\psi^m dxdt\le\\ &\dyle C_1 \bigg(\iint_{Q_T^2} |x|^{-\mu(\l)(p+1)}v^p\psi^m dxdt\bigg)^{\frac{1}{p}}
+C_2\bigg(\iint_{Q_T^2} |x|^{-\mu(\l)(p+1)}v^p\psi^m dxdt\bigg)^{\frac{1}{p}}.
\end{array}
\end{equation}
Therefore, using Young inequality, we obtain that
$$
\dyle \iint_{Q_T} |x|^{-\mu(\l)(p+1)}v^p\psi^m dxdt\le C,$$ where $C$ is independent of $R$ and $T$. Letting $R,T\to \infty$, we conclude that
$$
\int_0^\infty\int_{\ren}|x|^{-\mu(\l)(p+1)}v^p dxdt\le C,
$$
and the claim follows.

Recall that by \eqref{inter0} we have
\begin{equation}\label{V0}
\iint_{Q_T} |x|^{-\mu(\l)(p+1)}v^p\psi^m dxdt\le I+J,
\end{equation}
with
\begin{equation}\label{II}
I\le C\bigg(\iint_{Q_T^1} |x|^{-\mu(\l)(p+1)}v^p\psi^m dxdt\bigg)^{\frac{1}{p}},
\end{equation}
and
\begin{equation}\label{JJ}
J\le \dyle  m\iint_{Q_T^2}v\psi^{m-1} |L \psi| dxdt.
\end{equation}
From \eqref{II} and using the result of the claim we deduce that
\begin{equation}\label{V1}
I\le C\bigg(\iint_{\{R^2<t^2+|x|^{4s}<2R^2\}} |x|^{-\mu(\l)(p+1)}v^pdxdt\bigg)^{\frac{1}{p}}\to 0\mbox{  as  }R\to \infty.
\end{equation}
Now we deal with $J$. For $\kappa>0$ small enough, We have
\begin{eqnarray*}
J &\le & \dyle  m\iint_{Q_T^2}v\psi^{m-1} (1-\psi)^\kappa (1-\psi)^{-\kappa} |L \psi| dxdt\\ & \le & \bigg(\iint_{Q_T^2} |x|^{-\mu(\l)(p+1)}v^p\psi^m
(1-\psi)^\kappa dxdt\bigg)^{\frac{1}{p}} \bigg(\iint_{Q_T^2} |x|^{\frac{\mu(\l)(p+1)}{p-1}}\frac{|L
\psi|^{p'}}{\psi^{p'-m}(1-\psi)^{\kappa(p'-1)}}dxdt\bigg)^{\frac{1}{p'}}.
\end{eqnarray*}
Using the same change of variable as above we obtain that
$$
\iint_{Q_T^2} |x|^{\frac{\mu(\l)(p+1)}{p-1}}\frac{|L
\psi|^{p'}}{\psi^{p'-m}(1-\psi)^{\kappa(p'-1)}}dxdt=\iint_{\{\t^2+|y|^{4s}<2\}}|y|^{\frac{\mu(\l)(p+1)}{p-1}}\frac{|L
\theta(y,\t)|^{p'}}{\theta^{p'-m}(1-\theta)^{\kappa(p'-1)}}dyd\t=C_3.
$$
Thus
\begin{equation}\label{V2}
\begin{array}{lll}
J &\le & \dyle C \bigg(\iint_{Q_T^2} |x|^{-\mu(\l)(p+1)}v^p\psi^m (1-\psi)^\kappa dxdt\bigg)^{\frac{1}{p}}\\ &\le & \dyle C\bigg(\iint_{Q_T^1}
|x|^{-\mu(\l)(p+1)}v^p dxdt\bigg)^{\frac{1}{p}}\to 0 \mbox{  as  }R\to \infty.
\end{array}
\end{equation}
Thus combining \eqref{V0}, \eqref{V1} and \eqref{V2} and letting $R\to \infty$, we conclude that
$$
\int_0^\infty\int_{\ren}|x|^{-\mu(\l)(p+1)}v^p dxdt\le 0,
$$
a contradiction and then the result follows.
\end{proof}
\begin{remarks}
Notice that the above blow up result holds under the hypothesis that we can choose $\varphi\in \mathcal{C}^\infty_0(B_2(0))$ with $0\le \varphi\le 1$, $\varphi=1$
in $B_1(0)$ and
$$
\iint_{\{1<\t^2+|y|^{4s}<2\}} |y|^{-\mu(\l)}\t^{p'}\frac{|\varphi'(\t^2+|y|^{4s})|^{p'}}{\varphi^{p'-m}(\tau^2+|y|^{4s})}dyd\t\equiv C_1,
$$
$$
\iint_{\{1<\t^2+|y|^{4s}<2\}}|y|^{\frac{\mu(\l)(p+1)}{p-1}}\frac{|L \theta(y,\t)|^{p'}}{\theta^{p'-m}(1-\theta)^{\kappa(p'-1)}}dyd\t=C_3
$$
where $\theta(y,\t)=\varphi(\t^2+|y|^{4s})$. \\ The above conditions hold choosing $m$ closed to $p'$ and $\kappa$ small enough.
\end{remarks}

\section{Global existence for  $F(\lambda,s)<p<p_+(\l,s)$.}\label{global}
In order to show the optimality of $F(\l,s)$  we will prove that, under suitable condition on $u_0$, problem \eqref{Cauchy} has a global solution. To achieve this
affirmation, we will show the existence of a family of global supersolutions to problem \eqref{Cauchy} where $F(\l,s)<p<p_+(\l,s)$.

Recall that $F(\l,s)=1+\frac{2s}{N-\mu(\l)}$, since $p<p_+(\l,s)=1+\frac{2s}{\mu(\l)}$, then $\frac{2s}{p-1}>\mu(\l)$. Fix $\g>0$ be such that
 $\mu(\l)<\g<\frac{2s}{p-1}$, then for $T>0$, we define
\begin{equation}\label{www}
w(x,t,T)=A(T+t)^{-\theta}\big(\dfrac{|x|}{(T+t)^{\beta}}\big)^{-\gamma}H(\frac{|x|}{(T+t)^{\beta}}),
\end{equation}
where $\theta=\frac{2s}{p-1}$ and, as above, $\beta=\frac{1}{2s}$. Notice that
$$
w(x,t,T)=A(T+t)^{-\theta+\frac{\g}{2s}+\frac{N}{2s}}|x|^{-\gamma}h(x,t+T).
$$
It is clear that
$$
{h}_t(x,t+T)+(-\D)^s{h}(x,t+T)=0.
$$
We claim that, under suitable condition on $A$ and $T$, $w$ satisfies
\begin{equation}\label{eq:radialg}
w_t+(-\Delta)^s w-\lambda\dfrac{w}{r^{2s}}\geq w^p.
\end{equation}
For simplicity of typing we set
$$
D(x,t,T)=A(T+t)^{-\theta+\frac{\g}{2s}+\frac{N}{2s}}|x|^{-\gamma},
$$
then
$$
w(x,t,T)=D(x,t,T) h(x,t+T).
$$
By a direct computations we reach that
\begin{eqnarray*}
& w_t+(-\Delta)^s w-\lambda\dfrac{w}{r^{2s}}=\\ & D(x,t,T)\bigg({h}_t(x,t+T)+(-\D)^s {h}(x,t+T)\bigg)+{h}(x,t+T)\bigg(D_t(x,t,T)+(-\D)^s D(x,t,T)\bigg)\\ &-\dyle
\irn\dfrac{({h}(x,t+T)-{h}(y,t+T))(D(x,t,T)-D(y,t,T))}{|x-y|^{N+2s}}dy\\ &-\l \dfrac{D(x,t,T){h}(x,t+T)}{|x|^{2s}}.
\end{eqnarray*}
Since $T>0$, then
$$
{h}_t(x,t+T)+(-\D)^s {h}(x,t+T)=0.
$$
On the other hand we have
\begin{eqnarray*}
D_t(x,t,T)+(-\D)^s D(x,t,T)=(-\theta+\frac{\g}{2s}+\frac{N}{2s})\frac{D(x,t,T)}{(T+t)} + \l(\gamma)\frac{D(x,t,T)}{|x|^{2s}}.
\end{eqnarray*}
Since $\g>\mu(\l)$, then $\l(\g)>\l$.

We deal now with the mixed term
$$
J(x):=-\irn\dfrac{({h}(x,t+T)-{h}(y,t+T))(D(x,t,T)-D(y,t,T))}{|x-y|^{N+2s}}dy.
$$
By a direct computations, it follows that
\begin{eqnarray*}
J(x) &= & -A(T+t)^{-\theta+\frac{\g}{2s}}\irn\frac{(|x|^{-\gamma}-|y|^{-\g})(H(\frac{|x|}{(T+t)^{\beta}})-H(\frac{|y|}{(T+t)^{\beta}}))}{|x-y|^{N+2s}}dy\\ &=&
-AC^{N-\gamma}(T+t)^{-\theta+\frac{\g}{2s}-1}\irn\frac{(|x_1|^{-\gamma}-|y_1|^{-\g})(H(|x_1|)-H(|y_1|))}{|x_1-y_1|^{N+2s}}dy_1,
\end{eqnarray*}
where $x_1=\frac{|x|}{(T+t)^{\beta}}$ and $y_1=\frac{|y|}{(T+t)^{\beta}}$. Since $H$ is decreasing then $J(x)\ge 0$. Therefore, combining the above estimates it
holds that

\begin{eqnarray*}
w_t+(-\Delta)^s w-\lambda\dfrac{w}{r^{2s}} &=& (-\theta+\frac{\g}{2s}+\frac{N}{2s})\frac{w(x,t,T)}{(T+t)}+(\l(\gamma)-\l)\frac{w(x,t,T)}{|x|^{2s}}+J(x)\\ &\ge &
(-\theta+\frac{\g}{2s}+\frac{N}{2s})\frac{w(x,t,T)}{(T+t)}+(\l(\gamma)-\l)\frac{w(x,t,T)}{|x|^{2s}}.
\end{eqnarray*}
Hence, $w$ is a supersolution to \eqref{Cauchy} if we can chose $A,C>0$ such that
$$
(-\theta+\frac{\g}{2s}+\frac{N}{2s})\frac{w(x,t,T)}{(T+t)}+(\l(\gamma)-\l)\frac{w(x,t,T)}{|x|^{2s}}\ge w^p
$$
hence
$$
(-\theta+\frac{\g}{2s}+\frac{N}{2s})\frac{1}{(T+t)}+(\l(\gamma)-\l)\frac{1}{|x|^{2s}}\ge w^{p-1}.
$$
The last inequality is equivalent to have
\begin{equation}\label{s11}
\begin{array}{lll}
& (\dfrac{N+\g}{2s}-\theta)+ (\l(\gamma)-\l)(T+t)|x|^{-\gamma-2s}\\ &\,\\ &\ge A^{p-1}(T+t)^{-(p-1)\theta+\frac{(p-1)\g}{2s}+1}
|x|^{-(p-1)\gamma}H^{p-1}\Big(\dfrac{|x|}{(T+t)^{\beta}}\Big).
\end{array}
\end{equation}
Recall that $\theta=\frac{2}{p-1}$, since $\g<\frac{2s}{p-1}$, then
$$
-(p-1)\theta+\frac{(p-1)\g}{2s}+1=\frac{(p-1)\g}{2s}-1<0.
$$
On the other hand we have $2s+\g>(p-1)\g$. Thus going back to \eqref{s11}, we conclude that, for any $T>0$, we can choose $A$ small such that $w$ is a
supersolution to \eqref{Cauchy} and then the claim follows.

We are now able to state the main global existence result in this section.
\begin{Theorem}\label{globalm}
Assume that $F(\lambda,s)<p<p_+(\l,s)$. Let $u_0$ be a nonnegative function such that
$$
|x|^{\mu(\l)} u_0(x)\le \frac{C}{(1+|x|^2)^{\frac{N+2s}{2}}},
$$
then the Cauchy problem \eqref{Cauchy} has a global solution $u$ such that $u(x,t)\le w(x,t,T)$ for all $(x,t)\in \ren\times (0,\infty)$.
\end{Theorem}
\begin{proof}
Let $u_0$ be a nonnegative function such that the above condition holds, then $u_0\in L^2(\ren)$. According to the definition of $w$ given in \eqref{www}, there
exist $A,T>0$ such that $u_0(x)\le w(x,0,T)$ for all $x\in \ren$. Thus $w$ is a supersolution to problem \eqref{Cauchy}. Since $v(x,t)=0$ is a strict subsolution,
then using a classical iteration argument, the existence result follows.
\end{proof}

\begin{remark}

In the general case $1<p<p_+(\l)$ and under some hypotheses on $u_0$  it  is possible to show a complete blow-up in a suitable sense.

Suppose that $u_0(x)\ge h$ where $h\ge 0$ satisfies $h\in \mathcal{C}^{\infty}_0(\ren)$, $\text{supp}(h)\subset B_0(R)$ and
\begin{equation}\label{ddd}
\frac{1}{p+1}\irn h^{p+1}dx > \frac{a_{N,s}}{4}\iint_{D_\Omega}{\frac{(h(x)-h(y))^2}{|x-y|^{N+2s}}}\, dx\, dy-\frac{\l}{2}\irn \frac{ h^2}{|x|^{2s}}dx.
\end{equation}
Then if $u$ is a positive solution to problem \eqref{Cauchy} we have
$$
\int_{B_R(0)}u^2(x,t)dx\to \infty \mbox{  as  }t\to T^*.
$$
We argue by contradiction. Suppose that the above conditions holds and that
\begin{equation}\label{eq:nono}
\sup_{t\in (0,T)}\int_{B_R(0)}u^2(x,t)dx\le M(T)<\infty .
\end{equation}
Let  $w$ be the unique positive solution to the problem
\begin{equation}\label{newap}
\left\{
\begin{array}{rcll}
w_t+ (-\D)^s w &=&\l\dfrac{\,w}{|x|^{2s}+1}+w^p & \inn B_R(0) \times (0,T(w)),\\ w(x,t)&=& 0 &\inn (\ren \backslash B_R(0))\times (0,T(w)),\\ w(x,0)&=&h(x)
&\mbox{ if }x\in B_R(0).
\end{array}
\right.
\end{equation}
It is clear that $w\in L^2(0,T(h); H^s_0(B_R(0)))\cap L^{\infty}(B_R(0)\times (0,T(w))$. Since $u$ is a supersolution to \eqref{newap}, then $w\le u$ and
therefore $T(w)=\infty$. Define {\it the energy} in time $t$,
$$E(t)= \frac{a_{N,s}}{4}\iint_{D_{B_R(0)}}{\frac{(w(x,t)-w(y,t))^2}{|x-y|^{N+2s}}}\,
dx\, dy -\frac{\l}{2}\int_{B_R(0)}\frac{w^2}{|x|^{2s}+1}dx-\frac{1}{p+1}\int_{B_R(0)}w^{p+1}dx.$$ By a direct computations, it follows that $\frac{d}{dt}E(t)=
-\langle w_t, w_t\rangle\le 0$. Taking into consideration the hypothesis on $h$, we conclude that $E(t)\le 0$ for all $t$. Hence
\begin{eqnarray*} \frac{d}{dt}\int_{B_R(0)}w^2(x,t)dx \geq C \Big(\int_{B_R(0)}w^2(x,t)dx\Big)^{\frac{p+1}{2}}.
\end{eqnarray*}
By integration, it holds
$$\dint_{B_R(0)}w^2(x,t)dx\to
\infty \hbox{  as } t\to T^*<\infty,$$
 a contradiction with \eqref{eq:nono}.
\end{remark}

\begin{remark}
Notice that
$$
p_{-}(\lambda,s)=1+\frac{2s}{\bar{\mu}(\lambda)}\geq 1+\frac{2s}{N-\mu(\lambda)}.
$$
Hence,
$$
1+\dfrac{2s}{N}\leq 1+\dfrac{2s}{N-\mu(\lambda)}\leq p_-(\lambda,s)\leq 2_{s}^*-1\leq p_{+}(\lambda,s).
$$
See Figure 2.
\end{remark}

\

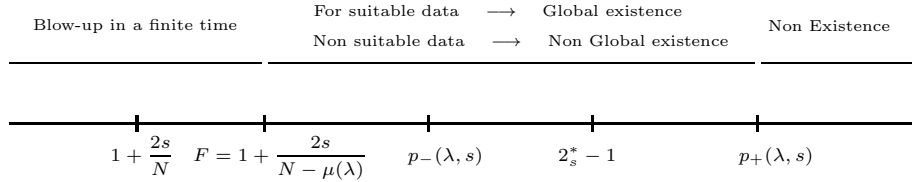
\begin{figure}
\begin{center}
\setlength{\unitlength}{1mm}
\begin{picture}(120,20)
\put(0,10){\line(1,0){33.5}} \put(34.5,10){\line(1,0){64.5}} \put(100,10){\line(1,0){20}} \thicklines \put(0,2){\line(1,0){120}} \put(17,1){\line(0,1){2}}
\put(34,1){\line(0,1){2}} \put(55.8,1){\line(0,1){2}} \put(77.6,1){\line(0,1){2}} \put(99.5,1){\line(0,1){2}} \put(16.75,15){\makebox(0,0){\tiny Blow-up in a
finite time}} \put(65,17){\makebox(0,0){\tiny For suitable data $\quad\longrightarrow\quad$ Global existence}} \put(68,13){\makebox(0,0){\tiny Non suitable data
$\quad\longrightarrow\quad$ Non Global existence}} \put(109,15){\makebox(0,0){\tiny Non Existence}} \put(24.5,-3){\scriptsize $F=1+\dfrac{2s}{N-\mu(\lambda)}$}
\put(13.5,-3){\scriptsize $1+\dfrac{2s}{N}$} \put(53,-3){\scriptsize $p_-(\lambda,s)$} \put(73,-3){\scriptsize $2_{s}^*-1$} \put(97,-3){\scriptsize
$p_+(\lambda,s)$}
\end{picture}
\end{center}
\caption{Existence versus blow-up}
\end{figure}

\end{document}